\documentclass[11pt, a4paper]{amsart}
\usepackage{amsmath}
\usepackage{amssymb}
\usepackage{amsthm}
\usepackage{graphicx}
\usepackage[hyperindex,backref]{hyperref}
\usepackage{caption}
\usepackage{subcaption}
\usepackage{palatino}
\usepackage[mnsy,cmsy,abx]{MnSymbol}

\newtheorem{theorem}{Theorem}[section]
\newtheorem{lemma}[theorem]{Lemma}
\newtheorem{proposition}[theorem]{Proposition}
\newtheorem{corollary}[theorem]{Corollary}

\theoremstyle{definition}
\newtheorem{definition}[theorem]{Definition}
\newtheorem{remark}[theorem]{Remark}

\newtheorem{convention and reminder}[theorem]{Convention and Reminder}
\newtheorem{convention and remark}[theorem]{Convention and Remark}
\newtheorem{definition and remark}[theorem]{Definition and Remark}

\newtheorem{reminders and definition}[theorem]{Reminders and Definition}

\newtheorem{notation and remarks}[theorem]{Notation and Remarks}
\newtheorem{notation and remark}[theorem]{Notation and Remark}
\newtheorem{notation and reminder}[theorem]{Notation and Reminder}
\newtheorem{example}[theorem]{Example}
\newtheorem{problem}[theorem]{Problem}

\newcommand\Spec{\operatorname{Spec}}
\newcommand\Proj{\operatorname{Proj}}
\newcommand\Rad{\operatorname{Rad}}

\newcommand\Img{\operatorname{im}}

\numberwithin{equation}{section}

\begin{document}

\title[Local Bezout estimates]{Local Bezout estimates and multiplicities of 
parameter and primary ideals}


\author[Bo\u{d}a]{Eduard Bo\u{d}a}

\address{Comenius University, 
Faculty of Mathematics, Physics and Informatics, 
SK --- 842 48 Brati\-slava, Slovakia}

\email{eduard.boda@fmph.uniba.sk}

\author[Schenzel]{Peter Schenzel}

\address{Martin-Luther-Universit\"at Halle-Wittenberg,
Institut f\"ur Informatik, D --- 06 099 Halle (Saale),
Germany}

\email{peter.schenzel@informatik.uni-halle.de}


\thanks{The authors are grateful to the Slovakian Ministry of 
	Education (Grant No. 1/0730/09), DAAD and Martin-Luther University Halle-Wittenberg 
	for supporting this research.}

\begin{abstract} 
	Let $\mathfrak{q}$ denote an $\mathfrak{m}$-primary ideal of 
	a $d$-dimensional local ring $(A, \mathfrak{m}).$ Let $\underline{a} = a_1,\ldots,a_d 
	\subset \mathfrak{q}$ be a system of parameters. Then there is the following 
	inequality for the multiplicities 
	$c  \cdot e(\mathfrak{q};A)  \leq   e(\underline{a};A)$ 
	where $c$ denotes the product of the initial degrees of $a_i$ in the form ring 
	$G_A(\mathfrak{q}).$ The aim of the paper is a characterization of the equality as well 
	as a description of the difference by various homological methods via Koszul 
	homology. To this end we have to characterize when the sequence of initial 
	elements $\underline{a^{\star}} = a_1^{\star}, \ldots,a_d^{\star}$ is a 
	homogeneous system of parameters of $G_A(\mathfrak{q}).$ In the case of 
	$\dim A = 2$ this leads to results on the local Bezout inequality. 
	In particular, we give several equations for improving the classical 
	Bezout inequality to an equality. 
\end{abstract}

\maketitle
\section{Introduction}
Let $C, D \subset \mathbb{A}^2_k$ be  two affine plane curves with no 
components in common. Let $f, g \in k[x,y]$ denote their defining equations, 
i.e. $C = V(f)$ and $D = V(g).$ Suppose that $0 \in C \cap D.$ Let 
$A = k[x,y]_{(x,y)}$ denote the local ring at the origin. Then the local 
Bezout inequality in the plane says that 
\[
e(f,g;A)\geq c\cdot d,
\]
where $e(f,g;A) $ denotes the local intersection multiplicity 
and $c$ and $d$ denote the (initial) degrees of $C$ and $D$ respectively. 
Equality holds if and only if $C$ and $D$ intersect transversally in $0,$ i.e. 
if and only if the initial forms $f^{\star}, g^{\star} \in k[X,Y]$ form a homogeneous system 
of parameters in $k[X,Y].$ This is a classical result, see \cite{BK} or \cite{Fi} for 
references. 

One of the aims of the present paper is the following Theorem. 

\begin{theorem} \label{0.1} With the previous notation there are the 
following results: 
\begin{itemize}
\item[(a)] $e(f,g;A) = c\cdot d +t +\ell,$ where $t$ denotes the number of common 
tangents in $(0,0)$ counted with multiplicities and $\ell$ is a non-negative number defined 
in local data.
\item[(b)] $e(f,g;A) = c \cdot d + e(f_1,g_1;A[x/y]) + 
e(f_2,g_2;A[y/x])-e(f_1,g_1;A[x/y,y/x])$, where $f_i,g_i, i = 1,2,$ denote the 
corresponding strict transforms of $f,g$ in the blowing up rings $A[x/y]$ and  $A[y/x]$. 
\item[(c)] Suppose $C$ and $D$ do not intersect transversally in the origin. Then 
\[
e(f,g;A) \leq c \cdot d + e(f_1,g_1;A[x/y]) + e(f_2,g_2;A[y/x])
\] 
with equality if and only if one of the coordinate axes is a common tangent in $(0,0)$. 
\end{itemize}
\end{theorem}

For the precise notion of $e$ we refer to  Remark \ref{7.0a}  (B).
The proof of Theorem \ref{0.1} is given in Theorems \ref{6.1},  \ref{7.2} and \ref{7.4}. 
The inequality  $e(f,g;A) \geq c\cdot d +t$ was proved by Byd\u{z}ovsk\'y (see 
\cite{By}) through the study of resultants. His result was one of the motivations 
for the investigations in the present paper. The formula in Theorem \ref{0.1} (b) was 
inspired by those of Greuel, Lossen and Shustin (see \cite[Proposition 3.21]{GLS}). 
In fact, we correct their formula by showing that it 
depends upon the embedding $C, D \subset \mathbb{A}^2_k$ in contrast to the claim in the proof 
of \cite[Proposition 3.21]{GLS}.

Another motivation for the authors was 
the paper \cite{Pr}. Let $\underline{a} = a_1,\ldots,a_d$ be a system of parameters 
in the local ring $(A,\mathfrak{m})$. Let $\mathfrak{q}$ denote an $\mathfrak{m}$-primary 
ideal with $(\underline{a}) \subseteq \mathfrak{q}$ and $a_i \in \mathfrak{q}^{c_i} 
\setminus \mathfrak{q}^{c_i+1}, i = 1,\ldots,d.$ Then $c_1\cdots c_d \; e(\mathfrak{q};A) 
\leq e(\underline{a};A)$ (see Lemma \ref{2.1}). In the case of $\mathfrak{q} = \mathfrak{m}$ 
it was claimed in \cite{Pr} that equality holds if and only if the sequence of initial 
elements $a_1^{\star},\ldots,a_d^{\star}$ forms an $G_A(\mathfrak{q})$-regular sequence. 
This is not true (see the Examples  \ref{2.2}).  Therefore we investigate the relation between both of these 
multiplicities. 
  
\begin{theorem} \label{0.2} Let $\mathfrak{q} \subset A$ denote an $\mathfrak{m}$-primary 
ideal. Let $\underline{a} = a_1,\ldots,a_d \subset \mathfrak{q}$ be a system of parameters and 
$a_i \in \mathfrak{q}^{c_i} \setminus \mathfrak{q}^{c_i+1}, i = 1,\ldots,d.$ 
\begin{itemize}
\item[(a)]  $e(\underline{a};A) = c_1\cdots c_d \, e(\mathfrak{q};A) + \chi(\underline{a},\mathfrak{q})$ 
for a certain  non-negative Euler characteristic $\chi(\underline{a},\mathfrak{q}).$ 
\item[(b)] If the sequence of initial  elements $\underline{a^{\star}} = a_1^{\star},\ldots,a_d^{\star}$ is a system 
of parameters in the form ring $G_A(\mathfrak{q}),$ then  $e(\underline{a};A) = c_1\cdots c_d \, e(\mathfrak{q};A).$ 
\item[(c)] The converse of the statement in (b) is true, provided $A$ is quasi-unmixed.
\end{itemize}
\end{theorem}

The investigation of the Euler characteristic $\chi(\underline{a},\mathfrak{q})$ is the main technical 
tool in order to prove the results in Theorems \ref{0.1} and \ref{0.2}. This Euler characteristic is defined 
in terms of a certain Koszul complex of the Rees algebra $R_A(\mathfrak{q}).$ To this end there are several 
investigations on Koszul homology modules. 
In Section 2 we study when the sequence of initial elements $\underline{a^{\star}} $ forms a
system of parameters in the form ring $G_A(\mathfrak{q}).$ This is of some independent interest. 

As a certain extension of Theorem \ref{0.2} we discuss the situation where 
$\underline{a^{\star}} = a_1^{\star},\ldots,a_d^{\star}$ is not necessarily a system of parameters in 
$G_A(\mathfrak{q})$. As a partial result we will be able to prove the following result.

\begin{theorem} \label{0.3} 
Let $(A,\mathfrak{m})$ denote a quasi-unmixed local ring. 
Let $\mathfrak{q} \subset A$ denote an $\mathfrak{m}$-primary 
ideal. Let $\underline{a} = a_1,\ldots,a_d \subset \mathfrak{q} = (q_1,\ldots,q_s)A$ be a system of parameters and 
$a_i \in \mathfrak{q}^{c_i} \setminus \mathfrak{q}^{c_i+1}, i = 1,\ldots,d.$ Suppose that  
$\underline{a^{\star}} = a_1^{\star},\ldots,a_d^{\star}$  
satisfies  
$\dim G_A(\mathfrak{q})/(\underline{a^{\star}}) = 1$. Then 
\[
e(\underline{a};A) = c_1\cdots c_d \, e(\mathfrak{q};A) + \sum_{i=1}^s (-1)^{i-1} \sum_{1 \leq j_1 < \ldots <j_i \leq s}
e(\underline{\tilde{a}}_{j_1}; A[\mathfrak{q}/q_{j_1}, \ldots, \mathfrak{q}/q_{j_i}]),
\]
where $\underline{\tilde{a}}_i = \tilde{a}_{1,i},\ldots,\tilde{a}_{d,i}$ denotes the sequence of 
strict transforms of $\underline{a} = a_1,\ldots,a_d$ on the blowing up rings $A[\mathfrak{q}/q_i], i = 1,\ldots,s$.
\end{theorem} 

The corresponding statement of Theorem \ref{0.3} for $\dim G_A(\mathfrak{q})/(\underline{a^{\star}}) > 1$ 
does not hold. That is, we get an expression 
of $\chi(\underline{a},\mathfrak{q})$ in the case of $\dim G_A(\mathfrak{q})/(\underline{a^{\star}}) 
\leq 1$. It is an open problem  how to go on in the remaining cases. A discussion in affine three space 
$\mathbb{A}^3_k$ is in preparation.

Section 3 is devoted to some motivating examples. In Section 4 we investigate the Euler 
characteristics related to certain Koszul complexes of the Rees algebra. In Section 5 we study 
the equality of the two multiplicities we are interested. The particular situation of dimension 2 
of the underlying ring is the contents of Sections 6 and 7.  The proof of Theorem \ref{0.2} is done in 
Theorems \ref{3.4}, \ref{4.1} and \ref{4.2}. In Section 8 we study the Euler characteristic 
$\chi(\underline{a},\mathfrak{q})$ in terms of the blowing up ring $R_A(\mathfrak{q})$ and the 
local cohomology of the \v{C}ech complex (see Theorem \ref{7.1} for the details). In the 
final Section 10 we illustrate the results by a few examples. 

In the terminology we follow Matsumura's textbook (see \cite{Ma}). For some basic 
results on the $\Proj$ of a graded ring we refer to \cite{HIO}.

\section{On Systems of Parameters} 
In the following let $(A,\mathfrak{m})$ denote a local Noetherian 
ring and $d = \dim A.$ Let $\mathfrak{q}$ be an $\mathfrak{m}$-primary 
ideal. Furthermore, let $\underline{a} = a_1,\ldots,a_d$ denote a system of 
parameters of $A.$ We write $(\underline{a})$ for the ideal generated 
by the elements $\underline{a}.$ 

Let $G_A(\mathfrak{q}) = \oplus_{n\geq 0} \mathfrak{q}^n/\mathfrak{q}^{n+1}$ 
denote the form ring of $A$ with respect to $\mathfrak{q}.$ The  
Rees ring is defined by $R_A(\mathfrak{q}) = \oplus_{n\in \mathbb N} 
\mathfrak{q}^nT^n \subset A[T].$ It follows that 
\[
R_A(\mathfrak{q})/\mathfrak{q}R_A(\mathfrak{q}) = G_A(\mathfrak{q}).
\]
Moreover, it is well-known that $\dim G_A(\mathfrak{q}) = \dim A$ and 
$\dim R_A(\mathfrak{q}) = \dim A +1.$ Now we assume that $\underline{a} 
\subseteq \mathfrak{q}.$ By the Krull Intersection Theorem for each 
$i \in \{1,\ldots,d\}$ there is a unique integer $c_i$ such that 
\[
a_i \in \mathfrak{q}^{c_i} \setminus \mathfrak{q}^{c_i +1} .
\] 
Moreover, for $i = 1,\ldots,d$ let 
\[
a_i^{\star} := a_i + \mathfrak{q}^{c_i+1}/\mathfrak{q}^{c_i+1} \in [G_A(\mathfrak{q})]_{c_i}
\]
denote the initial form of $a_i$ in $G_A(\mathfrak{q}).$ We 
define $c = c_1\cdots c_d$ and $e_i = c/c_i$ for $i = 1,\ldots,d.$  Then 
$a_i^{e_i} \in \mathfrak{q}^c$ and $(a_1^{e_1},\ldots,a_d^{e_d})  
\subseteq \mathfrak{q}^c.$ 

For the notion of a reduction of $\mathfrak{q}$ resp. a minimal reduction of 
$\mathfrak{q}$ we refer to \cite{Re} and \cite{SH}. Note that if for a system 
of elements $\underline{b} = b_1,\ldots,b_d$ of  $\mathfrak{q}$ the 
sequence of initial elements $\underline{b^{\star}} = b_1^{\star}, \ldots, b_d^{\star}$ 
is a homogeneous system of parameters in $G_A(\mathfrak{q}),$ then $\underline{b}$ 
is a system of parameters in $A.$ Here we need the following partial converse.

\begin{theorem} \label{1.1} With the previous notation the following conditions 
are equivalent: 
\begin{itemize}
\item[(i)] The ideal $(a_1^{e_1},\ldots,a_d^{e_d})A$ is a minimal reduction of 
$\mathfrak{q}^c.$ 
\item[(ii)] There is an integer $k$ such that $\mathfrak{q}^n = \sum_{i=1}^da_i \mathfrak{q}^{n-c_i}$ 
for all $n \geq k.$ 
\item[(iii)] The sequence $\underline{a}^{\star} = a_1^{\star},\ldots,
a_d^{\star}$ forms a system of parameters of $G_A(\mathfrak{q}).$
\end{itemize}
\end{theorem}

\begin{proof} 
First of all note that the condition (iii) is equivalent to the existence of 
an integer $k$ such that 
\[
[G_A(\mathfrak{q})/(\underline{a}^{\star})]_n  = \mathfrak{q}^n/(\sum_{i=1}^da_i \mathfrak{q}^{n-c_i} ,
\mathfrak{q}^{n+1}) = 0 \text{ for all } n \geq k.
\] 
By induction and Nakayama Lemma this is equivalent to the statement in (ii).  

It remains to prove the equivalence of the statements (i) and (ii). To this end 
we consider the commutative diagram of inclusions of graded $A$-algebras:
\[
\begin{array}{ccc}
A[a_1^{e_1}T^c, \ldots, a_d^{e_d}T^c] &  \subset  & A[a_1T^{c_1},\ldots,a_dT^{c_d}]\\ 
\cap &   & \cap \\ 
A[\mathfrak{q}^cT^c]  & \subset & A[\mathfrak{q}T].
\end{array} 
\] 
We have that $A[\mathfrak{q}^cT^c]  \simeq R_A(\mathfrak{q}^c)$ and $A[\mathfrak{q}T]
= R_A(\mathfrak{q}).$ First note that the two horizontal inclusions 
are integral as easily seen. Then both of the horizontal inclusions are 
finitely generated extensions (see \cite[Corollary 5.2]{AM}). Second we claim that the vertical inclusions 
are finitely generated extension if and and only if the condition (i) resp. the condition 
(ii) is fulfilled. This follows by a variant of the Artin-Rees Lemma (see \cite[Lemma 10.8]{AM}). 
Therefore the equivalence of (i) and (ii) follow by virtue of \cite[Proposition 2.16]{AM}. 
\end{proof} 

It is of  some interest to decide whether the system of parameters 
$\underline{a}^{\star} = a_1^{\star},\ldots,a_d^{\star}$ in $G_A(\mathfrak{q})$ is a 
$G_A(\mathfrak{q})$-regular sequence.  To this end we need the following definition:

\begin{definition} \label{1.1a} Let $\mathfrak{q}$ denote an $\mathfrak{m}$-primary ideal 
and $\underline{a} = a_1,\ldots,a_d$ a system of parameters contained in $\mathfrak{q}$ such that 
$\underline{a}^{\star} = a_1^{\star},\ldots,a_d^{\star}$ is a system of parameters in $G_A(\mathfrak{q})$. 
Let 
\[
f = f(\underline{a},\mathfrak{q})= \max \{n \in \mathbb{Z} | [G_A(\mathfrak{q})/(\underline{a}^{\star})]_n \not= 0\}
\]
denote the degree of nilpotency. 
Note that $f$ is a well-defined positive integer in case $\underline{a}^{\star}$ is a system 
of parameters in $G_A(\mathfrak{q}).$ Moreover define 
\[
a = a(\mathfrak{q}) = \max \{n \in \mathbb{Z} | [H^d_{\mathfrak{q}T}(G_A(\mathfrak{q}))]_n \not= 0\}
\]
the $a$-invariant of $\mathfrak{q}$. Here  $H^i_{\mathfrak{q}T}(G_A(\mathfrak{q}))$ denotes the $i$-th 
local cohomology of $G_A(\mathfrak{q}) = R_A(\mathfrak{q})/\mathfrak{q}R_A(\mathfrak{q})$ with respect 
to $\mathfrak{q}T$.  Clearly this is a finite number.  Note that $a(\mathfrak{q})$ is related to the $a$-invariant 
introduced by G\^{o}to and Watanabe (see \cite{GW}). 
\end{definition} 

\begin{theorem} \label{1.2}  Let $\underline{a} = a_1,\ldots,a_d$ be a system of parameters of 
the Cohen-Macaulay ring $(A,\mathfrak{m}).$ Let $\mathfrak{q} \supset (\underline{a})$ be an 
$\mathfrak{m}$-primary ideal. Then the following conditions are equivalent:
\begin{itemize}
\item[(i)] The sequence $\underline{a}^{\star} = a_1^{\star},\ldots,
a_d^{\star}$ is a $G_A(\mathfrak{q})$-regular sequence.
\item[(ii)] The sequence $\underline{a}^{\star} = a_1^{\star},\ldots,
a_d^{\star}$ is a system of parameters of $G_A(\mathfrak{q})$ and 
\[
(\underline{a}) \cap 
\mathfrak{q}^n = \sum_{i=1}^da_i \mathfrak{q}^{n-c_i} 
\]
for all $n =1,\ldots,f.$
\end{itemize} 
\end{theorem}

\begin{proof}
If $\underline{a^{\star}}$ is a $G_A(\mathfrak{q})$-regular sequence it is  a system of 
parameters since $d = \dim G_A(\mathfrak{q}).$ The relation  $(\underline{a}) \cap 
\mathfrak{q}^n = \sum_{i=1}^da_i \mathfrak{q}^{n-c_i} $ is true for all $n \geq 1$ as shown by 
Valabrega and Valla (see \cite[Corollary 2.7]{VV}). So the implication (i) $\Rightarrow$ (ii) is true. 

For the proof of the reverse implication we have to show that  $(\underline{a}) \cap\mathfrak{q}^n = 
\sum_{i=1}^da_i \mathfrak{q}^{n-c_i}$ is fulfilled  for all $n \geq 1$ (see \cite[Corollary 2.7]{VV}). 
By the definition of $f$ and since $\underline{a^{\star}}$ is a system of parameters it 
follows that 
\[
(\underline{a}) \cap \mathfrak{q}^n  = \mathfrak{q}^n = \sum_{i=1}^da_i \mathfrak{q}^{n-c_i} 
\text{ for all } n \geq f+1
\]
by Theorem \ref{1.1} and because of $\mathfrak{q}^{f+1} \subset (\underline{a})$ by the definition 
of $f = f(\underline{a}, \mathfrak{q})$.  Together with the assumption this completes the argument. 
\end{proof}

The advantage of the characterization of a regular sequence as
in Theorem \ref{1.2}   is its effectiveness 
for computational reasons. It is quite 
effective to check that the sequence $\underline{a}^{\star} = a_1^{\star},\ldots,
a_d^{\star}$ is a system of parameters. Then one has to check only finitely many 
equalities (depending on the number $f$) for the regularity of the 
sequence in $G_A(\mathfrak{q}).$ 

\begin{lemma} \label{1.3} With the previous notation suppose that $\underline{a}^{\star} = a_1^{\star},\ldots,a_d^{\star}$ 
is a system of parameters in $G_A(\mathfrak{q})$.  Then 
$
a(\mathfrak{q}) + \sum_{i=1}^d c_i \leq f(\underline{a},\mathfrak{q}).
$
If $G_A(\mathfrak{q})$ is a Cohen-Macaulay ring, equality holds.
\end{lemma}

\begin{proof} Since $a_1^{\star}$ is a parameter it is easily seen that there is an isomorphism 
\[
H^d_{\mathfrak{q}T}(G_A(\mathfrak{q})/ 0: a_1^{\star}) \simeq H^d_{\mathfrak{q}T}(G_A(\mathfrak{q})), 
d = \dim A.
\] 
Moreover  there is the following short exact sequence 
\[
0 \to (G_A(\mathfrak{q})/ 0: a_1^{\star})(-c_1) \stackrel{a_1^{\star}}{\to} G_A(\mathfrak{q}) \to 
G_A(\mathfrak{q})/(a_1^{\star}) \to 0.
\]
By the Grothendieck Vanishing Theorem it induces an exact sequence  
\[
H^{d-1}_{\mathfrak{q}T}(G_A(\mathfrak{q})/ (a_1^{\star})) \to H^d_{\mathfrak{q}T}(G_A(\mathfrak{q}))(-c_1) \to H^d_{\mathfrak{q}T}(G_A(\mathfrak{q})) \to 0.
\]
A simple argument implies $a +c_1 \leq 
\max \{n \in \mathbb{Z} | [H^{d-1}_{\mathfrak{q}T}(G_A(\mathfrak{q})/(a_1^{\star}))]_n \not= 0\}$. In case $G_A(\mathfrak{q})$ 
is a Cohen-Macaulay ring $a_1^{\star}$ is a regular element. That is, the first map in the previous exact 
sequence is injective and therefore equality holds.  By iterating this argument $d$-times it follows that 
$a +\sum_{i=1}^d c_i \leq 
\max \{n \in \mathbb{N} | [H^0_{\mathfrak{q}T}(G_A(\mathfrak{q})/(\underline{a}^{\star}))]_n \not= 0\}$. Because 
of  $H^0_{\mathfrak{q}T}(G_A(\mathfrak{q})/(\underline{a}^{\star})) \simeq G_A(\mathfrak{q})/(\underline{a}^{\star})$ the definition 
of $f$ proves the estimate. In case $G_A(\mathfrak{q})$ is a Cohen-Macaulay ring we get in each step equality. 
\end{proof}

\section{A Problem and an Example}
In this section we use the notation as it was introduced at the 
beginning of the previous one.  
Here we want to relate the multiplicity $e(\mathfrak{q},A)$ 
of $\mathfrak{q}$ with respect to $A$ to the multiplicity 
$e(\underline{a},A)$ of $(\underline{a})$ with respect to $A.$

For the definition of the multiplicity as well as other basic 
notions of commutative algebra we refer to Matsumura's book 
(see \cite{Ma}). As a first result of our investigations there is 
the following Lemma. 

\begin{lemma} \label{2.1} With the previous notation the following are true:
\begin{itemize}
\item[(a)]  
$c \cdot e(\mathfrak{q},A) \leq e(\underline{a};A),$
where $c = c_1\cdots c_d.$
\item[(b)] Equality holds if and only if $e(a_1^{e_1}, \ldots,a_d^{e_d};A) = e(\mathfrak{q}^c;A).$
\end{itemize}
\end{lemma} 

\begin{proof} 
For each $i = 1,\ldots,d$ we have that $a_i^{e_i} \in \mathfrak{q}^c,$ 
where $e_i = c/c_i.$ So, there is the following containment relation 
\[
(a^{e_1}_1,\ldots,a^{e_d}_d)A \subseteq \mathfrak{q}^c.
\] 
For the multiplicities this says  
\[
e(\mathfrak{q}^c;A) \leq e(a^{e_1}_1,\ldots,a^{e_d}_d;A),
\]
as easily seen. Now $e(\mathfrak{q}^c,A) = c^d \cdot e(\mathfrak{q},A)$ (see 
e.g. \cite[Formula 14.3]{Ma}). Moreover, $a^{e_1}_1,\ldots,a^{e_d}_d$ is  a system of parameters of $A,$ and therefore 
\[
e(\mathfrak{q}^c;A) \leq e(a^{e_1}_1,\ldots,a^{e_d}_d;A) = e_1 \cdots  e_d 
\cdot e(\underline{a};A).
\]
See \cite[Proposition 4.4]{AB} for the last equality. Because of $e_1 \cdots e_d = c^{d-1}$ 
the above relation proves the statements  of the Lemma. 
\end{proof}

In the case of $\mathfrak{q} = \mathfrak{m}$ Pritchard (see \cite[Lemma 3.1]{Pr}) 
claimed the following: $e(\underline{a};A) = c \cdot e(\mathfrak{m};A)$ holds if and 
only if $a^{\star}_1,\ldots,a^{\star}_d$ is a $G_A(\mathfrak{m})$-regular sequence. 
In particular, the equality $e(\underline{a},A) = c\cdot e(\mathfrak{m},A)$ implies that 
the form ring $G_A(\mathfrak{m})$ is a graded Cohen-Macaulay ring. This is not true 
as the following examples show. 

\begin{example} \label{2.2} (A) Let $k$ denote a field and $A = k[|t^4,t^5,t^{11}|] 
\subset k[|t|],$ where $t$ is an indeterminate over $k.$ Then $A$ is a one-dimensional 
domain and therefore a Cohen-Macaulay ring with $A \simeq k[|X,Y,Z|]/(X^4-YZ,Y^3-XZ,Z^2-X^3Y^2).$ 
Clearly, the residue class $a = x$ of $X$ is a parameter with $a \in \mathfrak{m} \setminus 
\mathfrak{m}^2,$ so that $c = 1.$  

Furthermore, by easy calculations it follows that $e(a,A) = \ell_A(A/aA) = 4$ and 
$e(\mathfrak{m},A) = 4.$ So, the equation $e(a,A) = c \cdot e(\mathfrak{m},A)$ holds, 
while $G_A(\mathfrak{m}) = k[X,Y,Z]/(XZ,YZ,Y^4,Z^2)$ is not a Cohen-Macaulay ring 
(see \cite[Section 6]{Va} for the details). 

(B) Let $A = k[|x,y,z|]/(x)\cap(y,z).$ Then $f= x^2-y, g=x^2-z$ forms a system of 
parameters of $A.$ It is easily seen that $e(f,g;A) = e(\mathfrak{m};A) =1.$ Moreover, 
$G_A(\mathfrak{m}) = k[X,Y,Z]/(X)\cap (Y,Z)$ and $f^{\star} = Y, g^{\star} = Z$ 
Therefore, $G_A(\mathfrak{m})$ is not a Cohen-Macaulay ring. Moreover, $f^{\star}, g^{\star}$ 
is not a system of parameters of $G_A(\mathfrak{m})$ as easily seen.
\end{example}

\section{The use of Koszul complexes} 
Another tool of our investigations is the Koszul complex. For the basic definitions 
and basic properties of it we refer to \cite{AB} or \cite{Ma}. In particular we use the 
Koszul complex $K(\underline{a};A)$ of $A$ with respect to $\underline{a} = a_1,
\ldots,a_d$ and the Koszul complex $K(\underline{aT};R_A(\mathfrak{q}))$ 
of $R_A(\mathfrak{q})$ with respect to the sequence $\underline{aT} = a_1T^{c_1},\ldots,
a_dT^{c_d}.$ It is a complex of $\mathbb{N}$-graded $R_A(\mathfrak{q})$-modules with 
homogeneous homomorphisms of degree zero.

\begin{definition} \label{3.1} First let $K(\underline{a};A)[T]$ be the $\mathbb N$-graded 
complex obtained by $K(\underline{a};A)$ in each degree. Then there is an 
embedding of complexes $K(\underline{aT};R_A(\mathfrak{q})) \subset K(\underline{a};A)[T]$ considered 
as complexes of $R_A(\mathfrak{q})$-modules which is homogeneous of degree zero. 
The co-kernel of this embedding is a complex, defined by $L(\underline{a};A).$ So there is a 
short exact sequence of complexes 
\[
0 \to K(\underline{aT};R_A(\mathfrak{q})) \to K(\underline{a};A)[T] \to 
L(\underline{a};A) \to 0.
\]
Let $n \in \mathbb N$ be an integer. By the restriction of the previous short exact sequence 
to the degree $n$  there is the following short exact sequence of complexes of $A$-modules 
\[
0 \to K(\underline{aT};R_A(\mathfrak{q}))_n \to K(\underline{a};A) \to 
L(\underline{a};A)_n \to 0. 
\]
Let $L_i(\underline{a};A)_n$ be the $n$-th graded component 
of the $i$-th module in the complex $L(\underline{a};A).$ By the definitions it is easily seen that 
\[
L_i(\underline{a};A)_n = \bigoplus_{1\leq j_1 < \ldots <j_i \leq d} 
A/\mathfrak{q}^{n-c_{j_1}- \cdots c_{j_i}} 
\]
with the boundary maps induced by the Koszul complexes, where 
$\mathfrak{q}^m = 0$ for $m \leq 0.$ 
\end{definition} 

\begin{lemma} \label{3.2} With the previous notation let $H_i(\underline{aT};R_A(\mathfrak{q}))_n$ 
denote the $n$-th graded component of the $i$-th homology module of $K(\underline{aT};R_A(\mathfrak{q})).$ 
Then $H_i(\underline{aT};R_A(\mathfrak{q}))_n$ is of finite length as an 
$A$-module for each $i$ and each $n.$
\end{lemma} 

\begin{proof} Take the second short exact sequence of complexes of $A$-modules as introduced 
in Definition \ref{3.1}. The long exact 
homology sequence induces the exact sequence 
\[
\ldots \to H_i(\underline{aT};R_A(\mathfrak{q}))_n \to H_i(\underline{a};A) \to 
H_i(L(\underline{a};A))_n \to \ldots
\] 
Since $\underline{a} = a_1,\ldots,a_d$ is a system of parameters of $A$ the homology module 
$H_i(\underline{a};A)$ 
is an $A$-module of finite length (see e.g. \cite{Ma}). Moreover, $H_i(L(\underline{a};A))_n$ 
is by definition a sub quotient of $L_i(\underline{a};A)_n.$ Since $\mathfrak{q}$ is an $\mathfrak{m}$-primary 
ideal it is of finite length too. So the short exact sequence provides the claim. 
\end{proof} 

\begin{definition} \label{3.3} Let $F :  0 \to F_r \to \ldots\to  F_1 \to F_0 \to 0$ denote 
a bounded complex of $A$-modules such that for all $i$ the homology module $H_i(F)$ is 
of finite length. Then define the Euler characteristic 
\[
\chi(F) = \sum_{i \geq 0} (-1)^i \ell_A(H_i(F))
\]
of $F.$ Let $\underline{a} = a_1,\ldots,a_d$ denote a system of parameters of $A.$ Then it is 
known by Serre and Auslander Buchsbaum (see \cite{AB} resp. \cite{S}) that the Euler characteristic 
$\chi(\underline{a};A)$ of the Koszul complex $K(\underline{a};A)$ coincides with the 
multiplicity $e(\underline{a};A)$ (see also Remark \ref{3.5}). 

By view of Lemma \ref{3.2} we are able to define the Euler characteristic of the $n$-th graded 
piece of the Koszul complex $K(\underline{aT};R_A(\mathfrak{q})).$ We call this 
$\chi(\underline{a},\mathfrak{q};n).$ That is, 
\[
\chi(\underline{a},\mathfrak{q};n) = \sum_{i\geq 0} (-1)^i \ell_A([H_i(\underline{aT};R_A(\mathfrak{q}))]_n).
\]
\end{definition} 

\begin{theorem} \label{3.4} With the previous notation we have the following results: 
\begin{itemize}
  \item[(a)] The Euler characteristic $\chi(\underline{a},\mathfrak{q};n) $ 
  is a non-negative constant, say $\chi(\underline{a},\mathfrak{q})$, for all $n\gg 0$.
  \item[(b)] $e(\underline{a};A) = c_1\cdots c_d \cdot e(\mathfrak{q};A) + \chi(\underline{a},\mathfrak{q}).$
\end{itemize}
\end{theorem}  

\begin{proof} Let us start with the short exact sequence of complexes of $A$-modules 
that was given in the Definition \ref{3.1}. All three complexes have 
homology of finite length. Therefore for each of them we may 
consider its Euler characteristic. By the additivity of the Euler characteristics 
on short exact sequences of complexes it follows that 
\[
\chi(\underline{a};A) = \chi(L(\underline{a};A)_n) + \chi(\underline{a},\mathfrak{q};n).
\]
Clearly $\chi(\underline{a};A) = e(\underline{a};A)$ as follows by the work of 
Auslander and Buchsbaum and Serre (see \cite{AB} and \cite{S}, respectively). Now 
let us continue 
with a calculation of the Euler characteristic of $L(\underline{a};A)_n.$ It is a well 
known fact that 
\[
\chi(L(\underline{a};A)_n) = \sum_{i \geq 0} (-1)^i \ell_A(L_i(\underline{a};A)_n)
\]
because all of the $L_i(\underline{a};A)_n$ are $A$-modules of finite length for all 
$i$ and all $n.$ By the structure of the $L_i(\underline{a};A)_n$  given in 
Definition \ref{3.1} it follows that
\[
\chi(L(\underline{a};A)_n) = \sum_{i \geq 0} (-1)^i (\sum_{1\leq j_1 < \ldots <j_i \leq d} 
\ell_A(A/\mathfrak{q}^{n-c_{j_1}- \cdots c_{j_k}})).
\]
The operation on the right side is the weighted $d$-fold backwards difference operator 
of the Hilbert-Samuel function $\ell_A(A/\mathfrak{q}^n).$ For $n \gg 0$ it is the weighted $d$-fold 
backwards difference operator of the Hilbert-Samuel polynomial. So the value is 
the constant 
\[
c_1\cdots  c_d \cdot e(\mathfrak{q};A).
\] 
This finishes the proof of the claim in (b). By the inequality shown in Lemma \ref{2.1} the proof of (a) 
is also complete. 
\end{proof}

\begin{remark} \label{3.5} The ideas of the proof of Theorem \ref{3.4} may be used in order 
to prove the result of Auslander-Buchsbaum and Serre that $\chi(\underline{a};A) = e(\underline{a};A)$ 
for a system of parameters $\underline{a} = a_1,\ldots,a_d$ of $A.$ To this end 
use the sum 
\[
\chi(\underline{a};A) = \chi(L(\underline{a};A)_n) + \chi(\underline{a},\underline{a};n)
\]
as shown in the proof of Theorem \ref{3.4}. Since $(\underline{aT})H_i(\underline{aT};R_A(\underline{a})) = 0$ 
for all $i \in \mathbb{Z}$ it follows that $[H_i(\underline{aT};R_A(\underline{a})]_n = 0$ for all 
$n \gg 0$ and all $i \in \mathbb{Z}.$ (Recall that $(\underline{aT})$ contains all elements 
of positive degree.) That is $\chi(\underline{a},\underline{a};n) = 0$ for 
all $n \gg 0.$ Therefore $\chi(\underline{a};A) = e(\underline{a};A)$ as follows because 
$\chi(L(\underline{a};A)_n) = e(\underline{a};A)$ for all $n \gg 0$.  

This argument was used also by Rees (see \cite[page 173]{R}). In fact Rees proved  a  more general 
result for  a good $\mathfrak{q}$-filtration on an $R$-module $M$. 
\end{remark}

\section{Characterization of equality}
We have seen in Lemma \ref{2.1} that $c \cdot e(\mathfrak{q},A) \leq e(\underline{a},A),
$ where $c = c_1\cdot \ldots \cdot c_d.$ In the next step we want to characterize the equality. 
We begin with a sufficient condition. 

\begin{theorem} \label{4.2} With the previous notation suppose that $\underline{a^{\star}} = a_1^{\star},\ldots,a_d^{\star}$ is a system of parameters of $G_{A}(\mathfrak{q}).$ 
Then $c_1 \cdots c_d \cdot e(\mathfrak{q},A) = e(\underline{a},A).$ In general the converse 
is not true.
\end{theorem} 

\begin{proof} Let $\underline{a^{\star}} 
= a_1^{\star}, \ldots,a_d^{\star}$ be a system of parameters of the form ring 
$G_A(\mathfrak{q})$. Whence the ideal $(a_1^{e_1}, \ldots,a_d^{e_d})$ 
is a minimal reduction 
of $\mathfrak{q}^c$ (see \ref{1.1}). Then $e(a_1^{e_1}, \ldots,a_d^{e_d};A) = 
e(\mathfrak{q}^c;A)$, which by Lemma \ref{2.1} implies that $c \cdot e(\mathfrak{q};A) 
= e(\underline{a};A)$. 
Finally, the Example \ref{2.2} (B) shows that the reverse implication does not hold in general. 
\end{proof}

In order to describe a necessary and sufficient condition we need further notation. 
For a local ring $A$ let $u_A(0)$ denote the intersection of those
primary components of the zero ideal that correspond to an associated prime ideal 
$\mathfrak{p}$ with $\dim A/\mathfrak{p} = \dim A.$ Moreover let $\hat{A}$ 
denote the completion of $A.$ In the following put $\overline{A} = \hat{A}/u_{\hat{A}}(0).$ 

\begin{theorem} \label{4.1}With the previous notation the following are equivalent: 
\begin{itemize}
  \item[(i)] $\underline{a^{\star}} = a_1^{\star},\ldots,a_d^{\star}$ is a system of parameters of $G_{\overline{A}}(\mathfrak{q}\overline{A}).$ 
  \item[(ii)] $c_1 \cdots c_d \cdot e(\mathfrak{q};A) = e(\underline{a};A).$ 
  \item[(iii)] The ideal $(a_1^{e_1},\ldots,a_d^{e_d})\overline{A}$ is a minimal reduction 
  of $\mathfrak{q}^c\overline{A}.$ 
\end{itemize}
\end{theorem} 

\begin{proof} 
First of all we note that $e(\underline{a};A) = e(\underline{a};\hat{A})$ and $e(\mathfrak{q};A) = 
e(\mathfrak{q};\hat{A})$ as easily seen. Second, by applying \cite[Theorem 14.7]{Ma} it follows 
that $e(\underline{a};\hat{A}) = e(\underline{a};\hat{A}/u_{\hat{A}}(0))$ and 
$e(\mathfrak{q};\hat{A}) = e(\mathfrak{q};\hat{A}/u_{\hat{A}}(0)).$ That is, without loss 
of generality we may assume that $A= \overline{A}.$ 

Then the equivalence of (i) and (iii) is shown in Theorem \ref{1.1}. By Lemma \ref{2.1} it follows that 
the condition (ii) is equivalent to the equality $e(a_1^{e_1},\ldots,a_d^{e_d};A) = e(\mathfrak{q}^c;A).$ 
Under the assumption that $A$ is quasi-unmixed  the last equality is equivalent to (iii). This 
is true by Rees' Theorem (see \cite{Re} or \cite[11.3.1]{SH}).
\end{proof}

\begin{corollary} \label{4.3} Let $(A,\mathfrak{m})$ denote a quasi-unmixed local ring. 
Then the following conditions are equivalent: 
\begin{itemize}
  \item[(i)] $\underline{a^{\star}} = a_1^{\star},\ldots,a_d^{\star}$ is a system of parameters of $G_A(\mathfrak{q}).$ 
  \item[(ii)] $c_1 \cdots c_d \cdot e(\mathfrak{q};A) = e(\underline{a};A).$ 
  \item[(iii)] The ideal $(a_1^{e_1},\ldots,a_d^{e_d})$ is a minimal reduction 
  of $\mathfrak{q}^c.$ 
\end{itemize}
\end{corollary} 

\begin{proof} It is a consequence of Theorem \ref{4.1}. To this end note that $u_{\hat{A}}(0) = 0$.
Moreover the statements in (i), (ii) and (iii) are equivalent to the corresponding statements for $\hat{A}$. To 
this end recall that $\hat{A}$ is a faithfully flat $A$-module.
\end{proof}

\section{The two dimensional situation} 
In this subsection we want to give an interpretation of the formula of Theorem 
\ref{3.4} (b) in the case of $\dim A = 2.$ To this end we modify our 
notation slightly. In this section we fix $(A, \mathfrak{m})$ a local 
two dimensional ring. Let $a,b$ denote a system of parameters of $A.$ As before let 
$\mathfrak{q}$ denote an $\mathfrak{m}$-primary ideal. We choose 
$c, d \in \mathbb N$ such that $a \in \mathfrak{q}^c \setminus \mathfrak{q}^{c+1}$ and 
$b \in \mathfrak{q}^d \setminus \mathfrak{q}^{d+1}.$ 

\begin{theorem} \label{5.1} We fix the above notation. Suppose that $a^{\star}$ 
is a $G_A(\mathfrak{q})$-regular element. Then the following hold:
\begin{itemize}
  \item[(a)]  The length $\ell_A((aA , \mathfrak{q}^n):_A b/(aA,\mathfrak{q}^{n-d}))$ 
  is a constant for all $n \gg 0.$ 
  \item[(b)] $\ell_A(A/(a,b)) = c\cdot d \cdot e(\mathfrak{q};A) +  
  \ell_A((aA , \mathfrak{q}^n):_A b/(aA,\mathfrak{q}^{n-d}))$ for all $n \gg 0.$ 
\end{itemize}
\end{theorem} 

\begin{proof} For an integer $n \in \mathbb N$ let $L(\underline{a},\mathfrak{q};n)$ 
denote the complex introduced in Definition \ref{3.1}.  In this particular situation 
it is the following 
\[
L(\underline{a},\mathfrak{q};n):\quad  0 \to A/\mathfrak{q}^{n-c-d} \stackrel{\alpha}{\longrightarrow} 
A/\mathfrak{q}^{n-d} \oplus A/\mathfrak{q}^{n-c} \stackrel{\beta}{\longrightarrow} A/\mathfrak{q}^{n} \to 0.
\]
The homomorphism $\alpha$ is given as 
\[
\alpha : r + \mathfrak{q}^{n-c-d} \mapsto (ra + \mathfrak{q}^{n-d}, rb + \mathfrak{q}^{n-c} )
\] 
while the homomorphism $\beta$ is defined by 
\[
\beta : (s + \mathfrak{q}^{n-d}, t + \mathfrak{q}^{n-c}) \mapsto sb - ta + \mathfrak{q}^n
\]
for all $n \in \mathbb N.$ 
The alternating sum of the length of all the modules of this complex gives (as indicated 
in the proof of Theorem \ref{3.4}) for all $n \gg 0$ the value $c\cdot d \cdot e(\mathfrak{q};A).$ 
Moreover, this coincides for all $n \gg 0$ with the Euler characteristic 
$\chi(\underline{a},\mathfrak{q};n).$ That is, 
\[
\chi(\underline{a},\mathfrak{q};n) = \ell_A(H_0(L(\underline{a},\mathfrak{q};n))) 
- \ell_A(H_1(L(\underline{a},\mathfrak{q};n))) + \ell_A(H_2(L(\underline{a},\mathfrak{q};n))). 
\]
In the next we shall calculate all of the homology modules that are involved. Clearly 
$H_0(L(\underline{a},\mathfrak{q};n)) \simeq A/((a,b)A , \mathfrak{q}^n).$ Since $\mathfrak{q}$ 
is an $\mathfrak{m}$-primary ideal and $a,b$ is a system of parameters it follows 
that $H_0(L(\underline{a},\mathfrak{q};n)) \simeq A/(a,b)A$ for all $n \gg 0.$ 
Furthermore, the second homology module is equal to 
\[
\ker \alpha = (\mathfrak{q}^{n-d}:_A a) \cap (\mathfrak{q}^{n-c} :_A b)/\mathfrak{q}^{n-c-d}.
\] 
This vanishes for $n \gg 0$ since $\mathfrak{q}^{n-d}:_A a =  
\mathfrak{q}^{n-c-d}.$ Recall that by our assumption $a^{\star}$  is 
a $G_A(\mathfrak{q})$-regular element.

Finally we investigate the first homology module. To this end we define  a 
homomorphism 
\[
\Phi : \ker \beta \to (aA , \mathfrak{q}^n):_A b/(aA,\mathfrak{q}^{n-d}) , \;
(s + \mathfrak{q}^{n-d}, t + \mathfrak{q}^{n-c}) \mapsto s + (aA , \mathfrak{q}^{n-d}).
\] 
Because $(s + \mathfrak{q}^{n-d}, t + \mathfrak{q}^{n-c}) \in \ker \beta$ 
it implies that $s \in (aA , \mathfrak{q}^n):_A b.$ Moreover $\Phi$ is surjective as 
$s + (aA ,\mathfrak{q}^{n-d} )\in  (aA, \mathfrak{q}^n ): b/(aA, \mathfrak{q}^{n-d})$ implies 
that $sb - ta \in \mathfrak{q}^n$ for a certain $t \in A.$ That is, 
$(s + \mathfrak{q}^{n-d}, t + \mathfrak{q}^{n-c}) \in \ker \beta.$ 

Now let us show that  $\ker \Phi = \Img \alpha.$ We have that 
$(s + \mathfrak{q}^{n-d}, t + \mathfrak{q}^{n-c}) \in \ker \beta$ if and only 
if $sb - ta \in \mathfrak{q}^n.$  Such an element belongs to $\ker \Phi$ if and only if 
$s- x a \in \mathfrak{q}^{n-d}$ for some $x \in A.$ 
Therefore $xab-ta \in \mathfrak{q}^n$ and $xb-t \in \mathfrak{q}^n :_A a = \mathfrak{q}^{n-c}.$ 
Here we used that $a^{\star}$ is $G_A(\mathfrak{q})$-regular. Finally that implies 
\[
(s + \mathfrak{q}^{n-d}, t + \mathfrak{q}^{n-c}) = (xa + \mathfrak{q}^{n-d}, xb + 
\mathfrak{q}^{n-c}) \in \Img \alpha.
\]
Because the converse is clear this finishes the proof of the isomorphism. 

So the Euler characteristic is completely described and therefore both of the statements  
are shown. 
\end{proof} 

For some  geometric applications in the next section we want to describe the length 
of $H_1(L(\underline{a},\mathfrak{q};n))$ for $n \gg 0$ in a different context.

\begin{proposition} \label{5.2} With the above notation suppose that $a^{\star}$ 
	is a regular element. Then it follows that 
\[
 \ell_A((aA , \mathfrak{q}^n):_A b/(aA,\mathfrak{q}^{n-d})) = \ell_A([a^{\star} G_A(\mathfrak{q}) : 
 b^{\star} /a^{\star} G_A(\mathfrak{q})]_{n-1-d}) + \ell_n,
 \]
 where $\ell_n = 
 \ell_A((aA,\mathfrak{q}^n):_A b/(aA,\mathfrak{q}^n):_A b\cap (aA, \mathfrak{q}^{n-d-1})).$
Moreover, all of the lengths are constants for all $n \gg 0.$ We put $\ell := \ell_n$ for  all $n \gg 0.$ 
\end{proposition}

\begin{proof} Because $a^{\star}$ is an $G_A(\mathfrak{q})$-regular element it follows that 
$G_A(\mathfrak{q})/(a^{\star}) \simeq G_{A/aA}(\mathfrak{q}A/aA)$ and therefore 
\[
[a^{\star} G_A(\mathfrak{q}) :  b^{\star} /a^{\star} G_A(\mathfrak{q})]_{n} \simeq 
(aA,\mathfrak{q}^{n+d+1}) : b \cap (aA, \mathfrak{q}^{n})/(aA,\mathfrak{q}^{n+1})
\] 
as it is easily seen. So there is the following short exact sequence 
\begin{gather*}
0 \to [a^{\star} G_A(\mathfrak{q}) :  b^{\star} /a^{\star} G_A(\mathfrak{q})]_{n-1-d} \to 
(aA , \mathfrak{q}^n):_A b/(aA,\mathfrak{q}^{n-d}) \to \\
(aA,\mathfrak{q}^n):_A b/(aA,\mathfrak{q}^n):_A b\cap (aA, \mathfrak{q}^{n-d-1}) \to 0.
\end{gather*}
By  counting the lengths we obtain the equality of the proposition. The length of the module 
in the middle is constant for $n \gg 0$ since  $G_{A/aA}(\mathfrak{q}A/aA)$ is of dimension one. 
By comparing the Hilbert polynomials this proves the final statement. 
\end{proof} 

\section{A first geometric application}
Let $k$ be an algebraically closed field. Let $C = V(f), D = V(g)$ be two plane curves without any 
common component. Let$f, g\in k[x,y]$ denote the defining 
equations. Suppose that $0 \in C\cap D.$ Let $A= k[x,y]_{(x,y)}$ denote the local 
ring at the origin. Moreover let $c,d$ denote the initial degree of $f$ and $g$ respectively. 
Then a classical result says that $e(f,g;A) \geq c\cdot d$ (see e.g. \cite[6.1]{BK} and \cite[8.6]{Fi}). The proof 
by Brieskorn and Kn\"orrer  needs resultants, while the proof by Fischer  uses Puiseux expansions. 
It is shown that equality holds if and only if 
$C$ and $D$ intersects transversally. That is, if and only the initial forms $f^{\star}, g^{\star}$ 
are a system of parameters in $G_A(\mathfrak{m}) = k[X,Y].$ 
In the following we want to present an improvement of this result mentioned by 
Byd\u{z}ovsk\'y (see \cite[Chap. XI, 134]{By}) and a further sharpening. 

\begin{theorem} \label{6.1} With the notation of the beginning of this section there is 
the following equality
\[
e(f,g;A) = c\cdot d + t + \ell_A((fA,\mathfrak{m}^n):_A g/(fA,\mathfrak{m}^n):_A g)\cap 
(fA, \mathfrak{m}^{n-d-1}))
\]
for all $n \gg 0,$ where $t$ denotes the number of common tangents counted with multiplicities. 
In particular $e(f,g;A) = c \cdot d + t + \ell \geq  c\cdot d + t$, where $\ell$ denotes 
the ultimative constant value of $\ell_A((fA,\mathfrak{m}^n):_A g/(fA,\mathfrak{m}^n):_A g)\cap 
(fA, \mathfrak{m}^{n-d-1}))$ for $n \gg 0$.
\end{theorem} 

\begin{proof} First note that $e(f,g;A) = \ell_A(A/(f,g)A).$ Therefore we apply the 
results of the previous sections. By  Theorem \ref{5.1} and  Proposition \ref{5.2} 
it will be enough to show that $t = \ell_A([f^{\star}B : g^{\star}/f^{\star}B]_n)$ for 
$n \gg 0,$ where $B = k[X,Y].$ We have $t = 0$ if and only if $f^{\star}, g^{\star}$ 
is a homogeneous system of parameters in $B,$ that is $C,D$ meet transversally in 
$0.$ Then $e(f,g;A) = c\cdot d$ (see  Theorem \ref{4.2}). So assume that $C$ and $D$ 
do not meet transversally. Then $f^{\star}, g^{\star}$ have a common factor $h$ and 
$f^{\star} = hr, g^{\star} = hs$ for two homogeneous polynomials $r, s \in B$ that are 
relatively prime. The degree $t :=\deg h$ denotes the number of common tangents 
counted with multiplicities. Then 
\[
f^{\star}B : g^{\star}/f^{\star}B \simeq rB/hrB \simeq B/hB(-u), \; u = \deg r.
\] 
Using this isomorphisms for $n \gg 0$ it follows that $\ell_A([f^{\star}B : g^{\star}/f^{\star}B]_n )
= t.$ This completes the proof of the statement. 
\end{proof}

The estimate $e(f,g;A) \geq  c\cdot d + t$ was proved by  Byd\u{z}ovsk\'y (see 
\cite[Kap. XI, 134]{By}). This is done by a study of resultants similar to the approach in 
\cite{BK}. Moreover it is not clear to the authors how to give 
a geometric interpretation of the correcting term in Theorem \ref{6.1}. We illustrate this 
investigations  with a few  examples (see Section 10).  

\begin{problem} \label{6.2} The authors do not know a geometric interpretation of 
the constant 
\[
\ell = \ell_A((fA,\mathfrak{m}^n):_A g/(fA,\mathfrak{m}^n):_A g)\cap 
(fA, \mathfrak{m}^{n-d-1})).
\]
This problem could be related to an interpretation of the integer $\ell$ as it was investigated in 
Proposition \ref{5.2} in homological terms.
\end{problem}

\section{The use of blowing ups} 
As above let $\mathfrak{q} = (q_1,\ldots,q_s)A$ denote an $\mathfrak{m}$-primary 
ideal of $A$ and $\underline{a} = a_1,\ldots,a_d, d = \dim A,$ a system of 
parameters contained in $\mathfrak{q}$. An affine covering of $\Proj R_A(\mathfrak{q})$ 
is given by $\Spec A[\mathfrak{q}/q_i], i = 1, \ldots,s$. The affine rings 
$A[\mathfrak{q}/q_i]$ are obtained as the degree zero components of the 
localizations $R_A(\mathfrak{q})_{q_iT}, i = 1,\ldots,s$. It follows that 
\[
R_A(\mathfrak{q})_{q_iT} \simeq A[\mathfrak{q}/q_i][q_iT,(q_iT)^{-1}]
\] 
(see \cite[Proposition 5.5.8]{SH} for the details). In the following we will examine the 
localization of the Koszul complex $K(\underline{aT};R_A(\mathfrak{q})_{q_iT}), i = 1,\ldots 
s$. For the notation of the strict transform we refer to \cite[(13.13)]{HIO}. 

\begin{lemma} \label{7.0} Fix the previous notation and assumptions. 
Then there are isomorphisms
\begin{gather*}
K(\underline{aT};R_A(\mathfrak{q})_{q_iT})_{\geq 0} \simeq K(\underline{\tilde{a}_i}; R_{A[\mathfrak{q}/q_i]}(\mathfrak{q}A[\mathfrak{q}/q_i]) 
\text{ and } \\
[K(\underline{aT};R_A(\mathfrak{q})_{q_iT})]_n \simeq 
K(\underline{\tilde{a}_i}; \mathfrak{q}^n A[\mathfrak{q}/q_i])  \text{ for all } n \geq 0 ,
\end{gather*}
where $\underline{\tilde{a}}_i = \tilde{a}_{1,i},\ldots,\tilde{a}_{d,i},$ denotes the sequence of 
strict transforms of $\underline{a} = a_1,\ldots,a_d$ on $A[\mathfrak{q}/q_i], i = 1,\ldots,s$.
\end{lemma}

\begin{proof} First of all we fix $j \in \{1,\ldots,s\}$ and put $q = q_j$.  
We consider the Koszul complex 
$K(\underline{aT};R_A(\mathfrak{q})_{qT})$. For the given $q = q_j$  it is - by the definition - the tensor product 
$$K(\underline{aT};R_A(\mathfrak{q})_{qT}) \cong \otimes_{i=1}^d K(a_iT^{c_i};R_A(\mathfrak{q})_{qT}).$$ For simplicity of 
notation put $a = a_i, c = c_i$. As a first step we claim that
\[ 
K(aT^c;R_A(\mathfrak{q})_{qT}) \simeq K(\tilde{a};R_A(\mathfrak{q})_{qT}),
\] 
where $\tilde{a} \in A[\mathfrak{q}/q]$ denotes the strict transform of $a\in A$.
To this end consider the following homomorphism of complexes
\[
\begin{array}{cccccccccl}
K(aT^c;R_A(\mathfrak{q})_{qT})& : & & 0 & \to& R_A(\mathfrak{q})_{qT}[-c] & \stackrel{aT^c}{\longrightarrow} & R_A(\mathfrak{q})_{qT} & \to & 0 \\
\downarrow                   &   & &   &    & \downarrow q^cT^c          &                                  &   \parallel           &     &   \\
K(\tilde{a};R_A(\mathfrak{q})_{qT})& : & & 0 & \to& R_A(\mathfrak{q})_{qT} & \stackrel{\tilde{a}}{\longrightarrow} & R_A(\mathfrak{q})_{qT} & \to & 0.
\end{array}
\]
Since $a \in \mathfrak{q}^c 
\setminus \mathfrak{q}^{c+1}$ the element $a$ may be written 
as a polynomial in $q_1,\ldots,q_s$ with initial degree $c$. Because of 
$q_j = q(q_j/q), j = 1,\ldots,s,$ in $A[\mathfrak{q}/q]$ 
we may substitute $q_j, j = 1,\ldots,s,$ in the element $a$. As a result of this process 
we get $aT^c = \tilde{a}q^cT^c$. The element $\tilde{a}\in A[\mathfrak{q}/q]$ is called the strict transform of $a \in A$ 
in $A[\mathfrak{q}/q]$, see also \cite[(13.13)]{HIO}. In fact, the above homomorphism of 
complexes is an isomorphism. By tensoring these isomorphisms we get the following isomorphism 
of Koszul complexes 
\[
K(\underline{aT};R_A(\mathfrak{q})_{qT}) \simeq K(\underline{\tilde{a}};R_A(\mathfrak{q})_{qT}),
\]
where $\underline{\tilde{a}} = \tilde{a}_1,\ldots \tilde{a}_d$ denotes the sequence 
of strict transforms of the elements $\underline{a} = a_1,\ldots,a_d$ in $A[\mathfrak{q}/q]$. 

As remarked above there is an isomorphism $R_A(\mathfrak{q})_{qT} \simeq A[\mathfrak{q}/q][qT,(qT)^{-1}].$  Now it is easy to see that the 
component of degree $n$ is given by $q^n A[\mathfrak{q}/q] = \mathfrak{q}^nA[\mathfrak{q}/q].$ 
Then the restriction to the $n$-th graded component 
of the Koszul complex $K(\underline{\tilde{a}};R_A(\mathfrak{q})_{qT})$ 
provides the claim. 
\end{proof} 

In order to give the Euler characteristic of $K(\underline{\tilde{a}_i}; \mathfrak{q}^n A[\mathfrak{q}/q_i]),$ in particular 
for $n = 0$,  the interpretation of a multiplicity we need a few explanations. 

\begin{remark} \label{7.0a} (A) Fix the previous notation. Then  $\dim A[\mathfrak{q}/q] \leq d$. 
Moreover it is not correct that $\dim A[\mathfrak{q}/q] = d$ for all $q = q_i, i = 1,\ldots,s$. To this end 
consider the Example \ref{2.2} (B). We have that $\dim A = 2$.  It follows  that $\dim A[\mathfrak{m}/x] = 1$,
 while $\dim A[\mathfrak{m}/y] = \dim A[\mathfrak{m}/z]  = 2$.  Suppose that $(A,\mathfrak{m})$ is 
 quasi-unmixed. Then 
it follows by the dimension formula (see \cite[Theorems 15.6 and 31.7]{Ma}) that $\dim A[\mathfrak{q}/q] = \dim A$ 
for all $\mathfrak{m}$-primary ideals $\mathfrak{q}$ and $q \in \mathfrak{q}$. \\
(B) Let $(A,\mathfrak{m})$ denote a quasi-unmixed local ring. Let $I \subset A$ denote an 
$\mathfrak{m}$-primary ideal such that $A[\mathfrak{q}/q]/(I)$ is of finite length. In order to define a multiplicity 
$e(I;A[\mathfrak{q}/q])$ - note that $A[\mathfrak{q}/q]$ is not a local ring - we proceed as follows: 
Because $A[\mathfrak{q}/q] $ is a Noetherian ring there are only 
finitely many maximal ideals $\mathfrak{m}_i, i = 1,\ldots,t,$ such that $IA[\mathfrak{q}/q] \subset \mathfrak{m}_i$. 
Because $A$ is quasi-unmixed it follows that $\dim A[\mathfrak{q}/q_i] = \dim A$ for all $i = 1,\ldots,t$.
By the Chinese Reminder Theorem it turns out that 
\[
A[\mathfrak{q}/q]/I^n \simeq \oplus_{i=1}^t A[\mathfrak{q}/q]_{\mathfrak{m}_i}/I^nA[\mathfrak{q}/q]_{\mathfrak{m}_i}
\]
for all $n \geq 0$. Therefore we define $e(I;A[\mathfrak{q}/q]) := \sum_{i=1}^t e(I;A[\mathfrak{q}/q]_{\mathfrak{m}_i})$.
In the following we will always use this notion for the multiplicity of the blowing up rings $A[\mathfrak{q}/q] $.  \\
(C)  In order to describe 
the multiplicity $e(\underline{\tilde{a}};A[\mathfrak{q}/q])$
in terms of Koszul homology (see \ref{3.5}) we need the assumption that $\underline{\tilde{a}} = \tilde{a}_1, \ldots, 
\tilde{a}_d$ is a system of parameters of $A[\mathfrak{q}/q]_{\mathfrak{m}_i}$ for all maximal ideals 
$\mathfrak{m}_i, i = 1,\ldots,t,$ containing $\underline{\tilde{a}} = \tilde{a}_1, \ldots, \tilde{a}_d$.
In general this is not the case. 
In the Example \ref{2.2} (B) we have  $\dim A[\mathfrak{m}/z] = 1$ while for the system of parameters 
$\{f,g\}$ the strict transforms consist of two elements. That is, we do need 
some additional assumption for our purposes here. 
\end{remark} 

\begin{lemma} \label{7.0b} Let $(A,\mathfrak{m})$ denote a quasi-unmixed local ring. 
With the previous notation let $\underline{a}^{\star} = a_1^{\star}, \ldots,a_d^{\star}$ denote 
the sequence of initial forms of $\underline{a} = a_1,\ldots,a_d$. Assume that $\dim G_A(\mathfrak{q})/(\underline{a}^{\star}) = 1$. The following are true:
\begin{itemize}
\item[(a)]
The factor ring $ A[\mathfrak{q}/q_i]/( \underline{\tilde{a}_i}),  i = 1,\ldots,s,$ is of finite length.
\item[(b)] There is at least one $i \in \{1,\ldots,s\}$ such that $( \underline{\tilde{a}_i})$ is a proper ideal. 
\item[(c)] If $( \underline{\tilde{a}_i}) \subset  A[\mathfrak{q}/q_i]$ is a proper ideal, it is a parameter ideal 
in $ A[\mathfrak{q}/q_i]_{\mathfrak{M}}$ for each maximal ideal $\mathfrak{M} \supseteq  ( \underline{\tilde{a}_i}) $.
\end{itemize}
\end{lemma}

\begin{proof} By the assumption there is an element $q \in \mathfrak{q}^c \setminus \mathfrak{q}^{c+1}$ which is a 
parameter of the ring $G_A(\mathfrak{q})/(\underline{a}^{\star})$. Therefore $\dim G_A(\mathfrak{q})/(\underline{a}^{\star},qT^c) = 0$.  This implies $[ G_A(\mathfrak{q})/(\underline{a}^{\star},qT^c)]_n = 0$ for all $n \gg 0$ and therefore the equalities 
\[
\mathfrak{q}^n = \sum_{i=1}^d a_i \mathfrak{q}^{n-c_i} +q \mathfrak{q}^{n-c} + \mathfrak{q}^{n+1}
\]
for all $n \gg 0$. By Nakayama Lemma it follows that $\mathfrak{q}^n = \sum_{i=1}^d a_i \mathfrak{q}^{n-c_i} +q \mathfrak{q}^{n-c}$ 
for all $n \gg 0$. That means that $R_A(\mathfrak{q})/(\underline{aT},q T^c)$ is of finite length and therefore 
$\dim R_A(\mathfrak{q})/(\underline{aT}) \leq 1$. But now $\Spec A[\mathfrak{q}/q_i]/( \underline{\tilde{a}_i}), 
i = 1,\ldots,s,$ is the affine covering of $\Proj R_A(\mathfrak{q})/(\underline{aT}))$ and therefore of dimension 
zero for all $i = 1,\ldots,s$.  This proves the claim in (a). 

Assume  that  $(\underline{\tilde{a}_i})$ is the unit ideal for all $i$. Then $\Proj R_A(\mathfrak{q})/(\underline{aT}) = \emptyset$ 
and $\dim R_A(\mathfrak{q}) \leq d$ which is a contradiction because $\dim A > 0$. This proves (b). 

We have that $\dim A[\mathfrak{q}/q_i]_{\mathfrak{M}} = d, i = 1,\ldots,s,$ for each 
maximal ideal $\mathfrak{M} \supseteq (\underline{\tilde{a}_i})$ (see \ref{7.0a}). Because $(\underline{\tilde{a}_i})$ is generated 
by $d$ elements the statement (a) proves the claim in (c). 
\end{proof} 

\begin{remark} \label{7.0c} Let $(A,\mathfrak{m})$ denote a quasi-unmixed local ring.  By arguments similar to those of Remark \ref{7.0a} it follows that  $\dim A[\mathfrak{q}/q_{j_1},\ldots, \mathfrak{q}/q_{j_i}] = d$ for all $1 \leq j_1 < \ldots < j_i \leq s$.    
By the definition of the $\Proj$ it holds that $(\underline{\tilde{a}_i}) A[\mathfrak{q}/q_i,\mathfrak{q}/q_j] = (\underline{\tilde{a}_j}) A[\mathfrak{q}/q_i,\mathfrak{q}/q_j]$ 
for all $i \not= j$. Now assume that in addition $\dim G_A(\mathfrak{q})/(\underline{a^{\star}}) = 1$.  As a consequence 
of \ref{7.0b} it follows that if $(\underline{\tilde{a}_k})A[\mathfrak{q}/q_{j_1},\ldots, \mathfrak{q}/q_{j_i}] $ 
is a proper ideal, then it is a parameter ideal in $A[\mathfrak{q}/q_{j_1},\ldots, \mathfrak{q}/q_{j_i}]_{\mathfrak{M}}$ for all $1 \leq j_1 < \ldots < j_i \leq s$ and $k \in \{j_1,\ldots,j_i\}$, where $\mathfrak{M}$ denotes a maximal ideal containing 
$(\underline{\tilde{a}_k})$. 
\end{remark}

The following remark will be the basic consideration for the computation of multiplicities in 
the next section. 

\begin{remark} \label{7.0d} We fix the previous notation. We will assume in 
addition that $(A,\mathfrak{m})$ is a quasi-unmixed local ring with $\dim A = d$ and $\dim G_A(\mathfrak{q})/(\underline{a^{\star}}) = 1$. 
Then (see \ref{7.0c}) $(\underline{\tilde{a}_k})$ is either 
the unit ideal or a parameter ideal in $A[\mathfrak{q}/q_{j_1},\ldots, \mathfrak{q}/q_{j_i}] _{\mathfrak{M}}$ 
for all $1 \leq j_1 < \ldots < j_i \leq s$ and $k \in \{j_1,\ldots,j_i\}$. So  $e(\underline{\tilde{a}_k}; A[\mathfrak{q}/q_{j_1},\ldots, \mathfrak{q}/q_{j_i}] )$ is equal to the 
Euler characteristic $\chi(\underline{\tilde{a}_k};A[\mathfrak{q}/q_{j_1},\ldots, \mathfrak{q}/q_{j_i}] )$. That is 
\[
e(\underline{\tilde{a}_k}; A[\mathfrak{q}/q_{j_1},\ldots, \mathfrak{q}/q_{j_i}] ) = 
\sum_{l=1}^d (-1)^l\ell (H_l({\underline{\tilde{a}_k}}; A[\mathfrak{q}/q_{j_1},\ldots, \mathfrak{q}/q_{j_i}] ))
\]
for  all $1 \leq j_1 < \ldots < j_i \leq s$ and $k \in \{j_1,\ldots,j_i\}$. This is easily seen by virtue 
of \ref{7.0a}, \ref{7.0b} and \ref{7.0c}.
\end{remark} 

\section{The use of local cohomology}

We want to study the \v{C}ech complex 
$C^{\cdot}$ of $R_A(\mathfrak{q})$ with respect 
to $\mathfrak{q}T = q_1T,\ldots,q_sT.$ That is, 
\[
(C^{\cdot},d^{\cdot}) : 0 \to R_A(\mathfrak{q}) \stackrel{d^0}{\to} \oplus_{i=1}^s R_A(\mathfrak{q})_{q_iT} 
\stackrel{d^1}{\to} \ldots \stackrel{d^{s-1}}{\to} R_A(\mathfrak{q})_{q_1T\cdots q_sT} \to 0
\]
(see \cite{SH} and \cite{Sc} for the details). It is a complex of graded $R_A(\mathfrak{q})$-modules. The $i$-th cohomology is the local cohomology module 
$H^i_{\mathfrak{q}T}(R_A(\mathfrak{q})),$ $ i \in \mathbb{Z}.$ It is a graded $R_A(\mathfrak{q})$-module such that the $n$-th graded component $[H^i_{\mathfrak{q}T}(R_A(\mathfrak{q}))]_n$ is a finitely generated $A$-module and  vanishes for all $n \gg 0.$  

\begin{lemma} \label{8.0} Let $(A,\mathfrak{m})$ denote a local ring. Let $\mathfrak{q} = (q_1,\ldots,q_s)A$ denote an 
$\mathfrak{m}$-primary ideal. Let $\underline{a} = a_1,\ldots,a_d$ denote a system of 
parameters with $(\underline{a}) \subset \mathfrak{q}$. With the previous notation it follows 
\[
[R_A(\mathfrak{q})_{q_{j_1}T \cdots q_{j_i}T}]_n = \mathfrak{q}^n  A[\mathfrak{q}/q_{j_1},\ldots, \mathfrak{q}/q_{j_i}] 
\simeq A[\mathfrak{q}/q_{j_1},\ldots, \mathfrak{q}/q_{j_i}] 
\] 
for all $n \geq 0$ and all $1 \leq j_1 < \ldots j_i \leq s$.
\end{lemma}

\begin{proof} An affine covering of $\Proj R_A(\mathfrak{q})$ 
is given by $\Spec A[\mathfrak{q}/q_i], i = 1, \ldots,s$. The affine rings 
$A[\mathfrak{q}/q_i]$ are obtained as the degree zero components of the 
localizations $R_A(\mathfrak{q})_{q_iT}, i = 1,\ldots,s$. It follows that 
$R_A(\mathfrak{q})_{q_iT} \simeq A[\mathfrak{q}/q_i][q_iT,(q_iT)^{-1}]$  
(see the beginning of the previous Section). By an iteration of the localization this 
provides that 
\[
R_A(\mathfrak{q})_{q_{j_1}T \cdots q_{j_i}T} = A[\mathfrak{q}/q_{j_1},\ldots, \mathfrak{q}/q_{j_i}] 
[q_{j_1}T,q_{j_1}^{-1}T^{-1}, \ldots,q_{j_i}T,q_{j_i}^{-1}T^{-1}]
\]
for all $1 \leq j_1 < \ldots j_i \leq s$. By considering the equality in degree $n \geq 0$ this proves the 
equality at the first. The second isomorphism follows since 
\[
\mathfrak{q}^n  A[\mathfrak{q}/q_{j_1},\ldots, \mathfrak{q}/q_{j_i}] = q_{j_1}^n  A[\mathfrak{q}/q_{j_1},\ldots, \mathfrak{q}/q_{j_i}] 
\]
and $q_{j_1}$ is regular on $A[\mathfrak{q}/q_{j_1},\ldots, \mathfrak{q}/q_{j_i}]$. 
\end{proof} 

As a consequence of Lemma \ref{8.0} we get  the degree $n$-component $n \geq 0$ of 
the  \v{C}ech complex. Note that it is defined by localizations.  

In  accordance with Lemma \ref{7.0} we will be able to 
examine the Koszul complex of the \v{C}ech complex and its Euler characteristics.

\begin{theorem} \label{7.1} Let $(A,\mathfrak{m})$ denote an unmixed local ring.  With the previous notation 
suppose that $\dim G_A(\mathfrak{q})/(\underline{a^{\star}}) =  1$. Then there is the equality 
\[
\chi(\underline{a},\mathfrak{q};n) = \sum_{i=1}^s (-1)^{i-1} \sum_{1 \leq j_1 < \ldots <j_i \leq s} 
e(\underline{\tilde{a}}_{j_1}; A[\mathfrak{q}/q_{j_1}, \ldots, \mathfrak{q}/q_{j_i}]),
\]
for all $n \gg 0$.
\end{theorem} 

\begin{proof} 
We use the \v{C}ech complex $C^{\cdot}$ of the beginning of this section. Note that 
\[
C^i \simeq \oplus_{1\leq j_1 < \ldots j_i \leq s} R_A(\mathfrak{q})_{q_{j_1} T\cdots q_{j_i}T}.
\]
Moreover the restriction of the complex $C^{\cdot}$ to the degree $n$ is exact in all degrees $n \gg 0$ 
since its cohomology modules vanish for all $n \gg 0.$ Let $d^i : C^i \to C^{i+1}$ denote the boundary 
map. From the complex we derive the following 
short exact sequences 
\[
0 \to \ker d^i \to C^i \to \Img d^i \to 0 \text{ and } 0 \to  \Img d^{i-1} \to \ker d^i \to H^i \to 0
\] 
for all $i \in \mathbb{Z}$. Here we use the abbreviation $H^i = H^i_{\mathfrak{q}T}(R_A(\mathfrak{q})).$
Next we apply the Koszul complex $K$ to the previous two short exact sequences. By Lemma \ref{7.0}, 
Lemma \ref{7.0b} and Remark \ref{7.0c} 
the $n$-th graded component of the  Koszul homology $H_l(\underline{aT},C^i)_n$ is of finite length 
for $l = 0,\ldots,d, i = 1,\ldots,s,$ and all $n \geq 0$. Moreover, the $n$-th graded component of the 
Koszul homology $H_l(\underline{aT},H^i)_n$ vanishes for all $n \gg 0.$  By view of the 
two short exact sequences induction on $i$ provides  that $H_l(\underline{aT}; \ker d^i)_n $ as well 
as $H_l(\underline{aT}; \Img d^i)_n$ are $A$-modules of finite length for all $l=0,\ldots,d,$ and all 
$n \gg 0$. 

The Koszul complex $K := K(\underline{aT};R_A(\mathfrak{q}))$ is a complex of free $R_A(\mathfrak{q})$-modules. 
Therefore it induces two short exact sequences of complexes 
\begin{gather*}
0 \to \ker d^i \otimes K \to C^i \otimes K \to \Img d^i \otimes K \to 0 \text{ and }\\
 0 \to  \Img d^{i-1} \otimes K \to \ker d^i \otimes K\to H^i \otimes K \to 0.
\end{gather*}
By the previous investigations we are able to evaluate the Euler characteristics of each of the complexes. 
By the additivity of Euler characteristics the first exact sequence yields that 
\[
\chi([C^i \otimes K]_n)=  \chi([\ker d^i \otimes K]_n) + \chi([\Img d^i \otimes K]_n)
\]
for all $n, i \in \mathbb{N}.$  By the same argument the second of these short exact sequence provides that 
\[
\chi([\Img d^{i-1} \otimes K]_n) = \chi([\ker d^i \otimes 
K]_n) \text{ for all } n \gg 0.
\]
To this end note that $\chi([H^i \otimes K]_n) = 0$ for all $n \gg 0.$ This follows  because of $H^i_n = 0$ 
for all $n \gg 0.$ Therefore $\chi([C^i \otimes K]_n)= \chi([\Img d^{i-1} \otimes K]_n) + \chi([\Img d^i \otimes K]_n)$ 
for all $n \gg 0$.

By applying the Koszul complex $K$ to the \v{C}ech complex $C^{\cdot}$  it provides the  
complex $C^{\cdot} \otimes K$. 
That is the single complex associated to the double complex
\[
0 \to C^0 \otimes K \to C^1 \otimes K \to \ldots \to C^s \otimes K \to 0.
\]
Now we claim that $\sum_{i=0}^s (-1)^i\chi([K\otimes C^i]_n) = 0$ for all $n \gg 0.$ This 
follows easily by summing up  the previous formulas for the Euler characteristics. 
In other words, by our definitions we get that 
$\chi(\underline{a},\mathfrak{q};n) = \sum_{i=1}^s (-1)^{i-1} \chi([K(\underline{aT};C^i)]_n)$ for all 
$n \gg 0$.  By virtue of Lemma \ref{7.0} and Remark \ref{7.0d} it follows that 
\[
\chi([K(\underline{aT};C^i)]_n) =  \sum_{1 \leq j_1 < \ldots <j_i \leq s}  
\chi(\underline{\tilde{a}}_{j_1,\ldots,j_i}; \mathfrak{q}^n A[\mathfrak{q}/q_{j_1},\ldots,\mathfrak{q}/q_{j_i}]),
\]
where $\underline{\tilde{a}}_{j_1,\ldots,j_i}$ denotes the strict transform of $\underline{aT}$ on 
$A[\mathfrak{q}/q_{j_1},\ldots,\mathfrak{q}/q_{j_i}]$. Next recall that 
This implies for the Euler characteristic of the Koszul homology 
\[
\chi(\underline{\tilde{a}}_{j_1,\ldots,j_i}; \mathfrak{q}^n A[\mathfrak{q}/q_{j_1},\ldots,\mathfrak{q}/q_{j_i}]) 
= e(\underline{\tilde{a}}_{j_1,\ldots,j_i}; A[\mathfrak{q}/q_{j_1}, \ldots, \mathfrak{q}/q_{j_i}])
\]
as follows by virtue of Auslander and Buchsbaum resp. by Serre (see \cite{AB} resp. \cite{S}). In order to simplify 
the formula recall that 
$(\underline{\tilde{a}_i}) A[\mathfrak{q}/q_i,\mathfrak{q}/q_j] = (\underline{\tilde{a}_j}) A[\mathfrak{q}/q_i,\mathfrak{q}/q_j]$ 
for all $i \not= j$. Therefore $(\underline{\tilde{a}}_{j_1,\ldots,j_i}) A[\mathfrak{q}/q_{j_1}, \ldots, \mathfrak{q}/q_{j_i}] = 
(\underline{\tilde{a}_k}) A[\mathfrak{q}/q_{j_1}, \ldots, \mathfrak{q}/q_{j_i}] $ for all $k \in \{j_1,\ldots,j_i\}.$ 
With this in mind and summing up all the direct summands of the \v{C}ech complex the additivity 
of the Euler characteristic provides the claim. 
\end{proof}

It should be mentioned that several of the multiplicities in the 
sum of Lemma \ref{7.1} might be zero. This happens for instance, if the strict transform in the corresponding 
ideal is the unit ideal. 

In the following we will apply the previous result to the local Bezout Theorem as studied in the 
previous section. To this end let  $C = V(f), D = V(g)$ two plane curves without any common 
component. Let$f, g\in k[x,y]$ denote the defining equations. Suppose that $0 \in C\cap D.$ 
Let $A= k[x,y]_{(x,y)}$ denote the local ring at the origin.

\begin{theorem} \label{7.2} With the notion of Section 7 it follows 
\[
e(f,g;A) = c\cdot d + e(f_1,g_1;A[\mathfrak{m}/x]) + e(f_2,g_2;A[\mathfrak{m}/y]) - e(f_1,g_1;A[\mathfrak{m}/x,\mathfrak{m}/y]),
\]
where $f_i,g_i, i = 1,2,$ denote the strict transform of $f,g$ on $A[\mathfrak{m}/x]$ and 
$A[\mathfrak{m}/y]$ respectively. 
\end{theorem}

\begin{proof} 
The proof is an immediate consequence of Theorem \ref{7.1}. Clearly $A$ as a regular local 
ring is quasi-unmixed. Moreover $G_A(\mathfrak{m}) = k[X,Y]$ and $0 \leq \dim G_A(\mathfrak{m})/(f^{\star},g^{\star}) \leq 1$. 
In the case of  $\dim G_A(\mathfrak{m})/(f^{\star},g^{\star}) = 2$ we have $e(f,g;A)= c\cdot d$ and all of the other 
multiplicities are zero. The case of  $\dim G_A(\mathfrak{m})/(f^{\star},g^{\star}) = 1$ is covered by 
Theorem \ref{7.1}. 
\end{proof}

The formula shown in Theorem \ref{7.2}  provides a correction  to the formula \cite[3.21]{GLS}. See also 
the discussion in the next Section. 

\section{Examples and a second geometric application} 
Let $C =V(f), D =V(g) \subset \mathbb{A}^2_k$ be the two plane curves defined by 
\[
f = x^3+y^3-3xy, g = x^2+y^2-2ax \in k[x,y],
\] 
where $a \in k$ is a constant. Here $k$ is an algebraically closed field. The curve 
$C$ is the "folium cartesium", while the curve $D$ is the circle with center $(a,0)$ and 
radius $a$ (see Figure \ref{folium}). We have that $(0,0) \in C \cap D.$ We consider the local intersection at the 
origin. To this end put $A = k[x,y]_{(x,y)}.$ Because $C$ and $D$ do not have a component 
in common $f,g$ is a system of parameters of $A.$ 
\begin{figure*}
	\centering
	\includegraphics[width=0.5\linewidth]{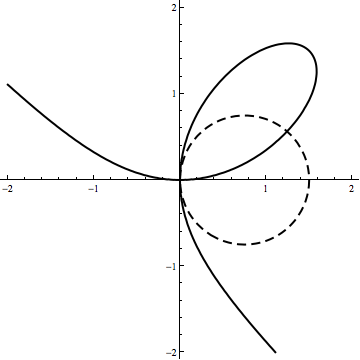}
	\caption{The intersection of "folium cartesium" with the circle for $a = 3/2.$}
	\label{folium}
\end{figure*}
First there is the computation of the local multiplicity $e(f,g;A)$ of $(0,0) \in C \cap D$.

\begin{proposition} \label{8.1} With the previous notation the multiplicity is given by 
\begin{equation*}
e(f,g;A) = 
\begin{cases} 
3 & \text{if $a \not= 0$ and $a \not= 3/2,$}\\
4 & \text{if $a = 0,$}  \\
5 & \text{if $a =3/2.$}
\end{cases}
\end{equation*}
\end{proposition}

\begin{proof} It is easy to see that $C \cap D = V(f',g),$ where $f' = x^3-xy((3-2a)+x).$ Then 
the result follows by some simple calculations. Note that for $a = 0$ the curve $D$ consists of 
one (real) point. But it is the union of two (conjugate) complex lines intersecting in the origin. 
\end{proof}

Finally let us summarize the correcting terms of the Bezout inequality as introduced  
in Theorem \ref{6.1} and Theorem \ref{7.2}. In respect to the definitions 
of Theorem \ref{7.2} we put 
\[
e_1 = e(f_1,g_1;A[x/y]), \; e_2 = e(f_2,g_2;A[y/x]), \; e_3 = e(f_1,g_1;A[x/y,y/x])
\]
where $f_i,g_i, i = 1,2,$ denote the corresponding strict transforms. 
Then we get the following result. 

\begin{proposition} \label{8.2} For the Bezout numbers of the intersection of the "folium 
cartesium" with the circle we have:
\[
\begin{array}{|c|cc|cc|ccc|}
\hline 
e & c & d & t & \ell & e_1 & e_2 &e_3 \\
\hline 
3 & 2 & 1 & 1 &  0      & 1 & 0 & 0 \\
4 & 2 & 2 & 0 & 0       & 0 & 0 & 0\\ 
5 & 2 & 1 & 1 & 2       & 3 & 0 & 0\\
\hline
\end{array}
\]
Here $e$ denotes the multiplicity. Moreover $\ell$ denotes the invariant  introduced in 
Theorem \ref{6.1}. Furthermore $t$ denotes the number of common tangents. 
\end{proposition}

\begin{proof} The multiplicity is computed in Proposition \ref{8.1}. It is obvious to verify 
the numbers $c,d,$ the initial degrees. In all the cases with multiplicity different from 
$4$ we have $t = 1.$  Therefore we have the value of $\ell$ by Theorem \ref{6.1}. 

In order to complete the table we have to calculate $e_i, i = 1,2.$ In the case of $a = 0$ 
there is nothing more to calculate.   So let us  assume that $a \not= 0.$ 
Because of $A[x/y] \simeq A[S]/(yS-x)$ it follows that 
\[
f_1 = yS^3+y-3S, \, g_1 = yS^2+y-2aS. 
\]
Then an easy calculation shows that $e_1 = 1$ provided $a \not= 0$ and $a \not= 3/2,$ 
while $e_1 = 3$ if $a = 3/2.$ 
Furthermore $A[y/x] \simeq A[T]/(xT-y)$ and 
\[
f_2 = x + xT^3-3T, \, g_2 =  x+ xT^2 -2a. 
\]
So that in all cases $e_2 = e_3 = 0.$
\end{proof}

In the following we shall investigate the dependence of $e(\underline{\tilde{a}};
A[\mathfrak{q}/q_{j_1},\ldots,\mathfrak{q}/q_{j_i}])$ on the particular choice of the basis 
$\mathfrak{q} = (q_1,\ldots,q_s)A$. We will investigate the most simple case, namely the geometric
situation with $A = k[x,y]_{(x,y)}$ as it was considered in Section 7. 

Let $C_i = V(F_i), D_i =V(G_i) \subset \mathbb{A}_k^2, i = 1,2$ the pair of two plane cubic curves defined by 
\[
\begin{array}{ll}
F_1 = x^3 - (x^2 - y^2), & G_1 = y^3 - (y^2 - x^2),\\ 
F_2 = (x + y)^3 - 4xy, & G_2 = (x - y)^3 + 4xy,
\end{array} 
\]
see Figure \ref{cubics}.
We have that $(0,0) \in C_i \cap D_i, i = 1,2$.  We consider the local intersection at the 
origin and  put $A = k[x,y]_{(x,y)}$.

\begin{figure}
        \centering
        \begin{subfigure}[b]{0.45\textwidth}
                \centering
                \includegraphics[width=\textwidth]{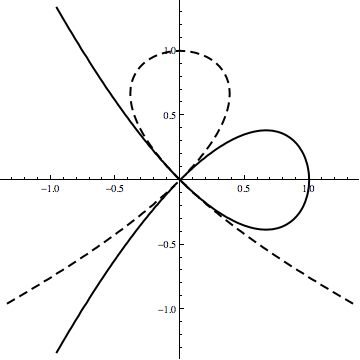}
                \caption{$C_1\cap D_1$}
                \label{cubic1}
        \end{subfigure}
~
        \begin{subfigure}[b]{0.45\textwidth}
                \centering
                \includegraphics[width=\textwidth]{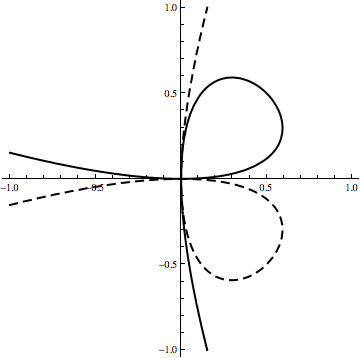}
                \caption{$C_2\cap D_2$}
                \label{cubic2}
        \end{subfigure}
        \caption{The intersection \ref{cubic2} is obtained by an affine transformation of the intersection \ref{cubic1}}\label{cubics}
\end{figure}

It comes out that the multiplicities of the blowing ups depend on the embedding in the affine space $\mathbb{A}_k^2$. 
That is, they depend on the particular choice of the basis of $\mathfrak{q}$. We illustrate this by the above 
examples. We put 
\[
e_1 = e(f_1,g_1;A[x/y]), \; e_2 = e(f_2,g_2;A[y/x]), \; e_3 = e(f_1,g_1;A[x/y,y/x])
\]
where $f_i,g_i, i = 1,2,$ denote the  strict transforms on $A[x/y]$ resp. $A[y/x]$. 
Then we get the following values:
\[
\begin{array}{|c|c|cc|ccc|}
\hline 
i & e & c & d &  e_1 & e_2 &e_3 \\
\hline 
1 & 7 & 2 & 2 &  3  & 3 & 3\\
2 & 7 & 2 & 2 & 2  & 1 & 0\\ 
\hline
\end{array}
\]
It is worth to remark that one can calculate the $e_i, i =1,2,3,$ easily 
with the aid of the Computer Algebra System {\sc Singular} (see \cite{DGPS}) also in more complicated 
examples.

\begin{remark} \label{7.3} In \cite[Proposition 3.21]{GLS} the authors claim a formula similar to those 
of Theorem \ref{7.2}. In their formula they get $e(\underline{a};A) = c\cdot d + e(f_1,g_1;A[\mathfrak{m}/x]) 
+ e(f_2,g_2;A[\mathfrak{m}/y])$. Here we use our notation.  In respect to the situation $C_1\cap D_1$ above we obtain an example 
with $e(f_1,g_1;A[\mathfrak{m}/x,\mathfrak{m}/y]) \not= 0$. In the proof of their result they assume 
that $g$ has as a tangent the $x$-axis. As the above examples show the  multiplicities of 
the blowing ups depend upon the concrete embedding.
\end{remark}

We conclude with a result on the vanishing of $e(f_1,g_1;A[\mathfrak{m}/x,\mathfrak{m}/y])$.

\begin{theorem} \label{7.4} Let $f,g \in k[x,y]$ denote the defining equations of two plane 
curves $C,D \subset \mathbb{A}^2_k$. We fix the notation of Section 7.
\begin{itemize}
\item[(a)] $e(f,g;A) = c\cdot d$ if and only if $e(f_1,g_1;A[x/y]) = e(f_2,g_2;A[y/x]) = 0$. That is, if and only if $C$ and $D$ 
intersect transversally in $(0,0)$.
\item[(b)] Assume that $C$ and $D$ do not intersect transversally in $(0,0)$. Then 
\[
e(f,g;A) \leq c\cdot d + e(f_1,g_1;A[x/y]) + e(f_2,g_2;A[y/x]).
\] 
Equality holds if and only if 
$C$ and $D$ have a coordinate axis as a common tangent in the origin.
\end{itemize}
\end{theorem}

\begin{proof} First let us prove the statement in (a). By view of \ref{4.1} it is known that 
$e(f,g;A) = c\cdot d$ if and only if $f^{\star}, g^{\star}$ is a system of parameters in 
$G_A(\mathfrak{m}) = k[X,Y]$. In other words $\Proj G_A(\mathfrak{m})/(f^{\star},g^{\star}) = 
\emptyset$. The last statement is equivalent to the fact that $(f_i,g_i), i= 1,2,$ generates 
the unit ideal on $A[x/y]$ and $A[y/x]$ respectively. This is easily seen equivalent to the 
vanishing of the multiplicities in the statement (a). 

In order to prove (b) we may assume $G_A(\mathfrak{m})/(f^{\star}, g^{\star}) = k[X,Y]/(f^{\star}, g^{\star})$ 
is of dimension one. Then $f^{\star}, g^{\star}$ have a common factor $h$ and 
$f^{\star} = hr, g^{\star} = hs$ for two homogeneous polynomials $r, s \in k[X,Y]$ that are 
relatively prime. The equation $h$ describes the common tangents of $C$ and $D$. 
On the other side $e(f_1,g_1;A[x/y,y/x]) = 0$ if and only if $(f_1,g_1) A[x/y,y/x] = 
(f_2,g_2) A[x/y,y/x]$ is the unit ideal. By the definition of the $\Proj$ this is equivalent 
to $(G_A(\mathfrak{m})/(f^{\star},g^{\star}))_{xTyT} = 0$. This is true if and only if $XY 
\in \Rad ((h) k[X,Y])$. Therefore if and only if one of the axis is a common tangent to 
$C$ and $D$ in the origin $(0,0))$. This finishes the proof of the claim in (b).
\end{proof}

It would be of some interest to find a relation between the statement in Theorem \ref{7.4} 
to Problem \ref{6.2}. 
\smallskip 

\noindent {\bf{Acknowledgement.}} The authors are grateful to the reviewer for a careful reading 
	of the manuscript.

\end{document}